\documentclass[a4paper,12pt,reqno]{amsart}
\usepackage[T1]{fontenc}
\usepackage[utf8]{inputenc}
\usepackage{bera}
\usepackage{amssymb,dsfont,mathrsfs}
\usepackage[margin=1in]{geometry}
\usepackage{enumitem}
\usepackage{graphicx,xcolor,tikz}
\usepackage[bookmarksdepth=2]{hyperref}

\DeclareFontFamily{T1}{fve}{}
\DeclareFontShape{T1}{fve}{m}{sc}{<->sub*fve/bx/n}{}

\numberwithin{equation}{section}

\newtheorem{theorem}{Theorem}[section]
\newtheorem{lemma}[theorem]{Lemma}
\newtheorem{corollary}[theorem]{Corollary}

\theoremstyle{definition}

\DeclareMathOperator{\sign}{sign}

\DeclareMathOperator{\ind}{\mathds{1}}

\newcommand{\sub}{\subseteq}

\newcommand{\R}{\mathds{R}}

\newcommand{\D}{\mathscr{D}}
\newcommand{\E}{\mathscr{E}}
\newcommand{\Ea}{\tilde{\E}}
\newcommand{\loc}{\mathrm{c}}
\newcommand{\jmp}{\mathrm{j}}
\newcommand{\kll}{\mathrm{k}}
\newcommand{\El}{\E^{\loc}}
\newcommand{\Eal}{\Ea^{\loc}}
\newcommand{\Eaj}{\Ea^{\jmp}}
\newcommand{\Eak}{\Ea^{\kll}}
\newcommand{\mul}{\mu^{\loc}}
\newcommand{\ph}{\varphi}
\newcommand{\eps}{\varepsilon}

\renewcommand{\le}{\leqslant}
\renewcommand{\ge}{\geqslant}

\newcommand{\pow}[1]{^{\langle #1\rangle}}
\newcommand{\abs}[1]{\lvert #1 \rvert}
\newcommand{\norm}[1]{\lVert #1 \rVert}
\newcommand{\ignore}[1]{}

\begin{document}

\title[Beurling--Deny formula for Sobolev--Bregman forms]{Beurling--Deny formula for Sobolev--Bregman forms}
\author[Michał Gutowski]{Michał Gutowski}
\author[Mateusz Kwaśnicki]{Mateusz Kwaśnicki}
\email{michal.gutowski@pwr.edu.pl, mateusz.kwasnicki@pwr.edu.pl}
\address{\normalfont Department of Pure Mathematics, Wroc\l{}aw University of Science and Technology, Wyb.\@ Wyspia\'nskiego 27, 50-370 Wroc\l{}aw, Poland.}
\thanks{M.~Gutowski was supported by the Polish National Science Centre (NCN) grant no.\@ 2018/31/B/ST1/03818. M.~Kwaśnicki was supported by the Polish National Science Centre (NCN) grant no.\@ 2019/33/B/ST1/03098}

\begin{abstract}
For an arbitrary regular Dirichlet form $\E$ and the associated symmetric Markovian semigroup $T_t$, we consider the corresponding Sobolev--Bregman form $\E_p(u) = -\tfrac{1}{p} \tfrac{d}{d t}\bigr\vert_{t = 0} \|T_t u\|_p^p$, where $p \in (1, \infty)$. We prove a variant of the Beurling--Deny formula for $\E_p$. As an application, we prove the corresponding Hardy--Stein identity. Our results extend previous works in this area, which either required that $\E$ is translation-invariant, or that $u$ is sufficiently regular.
\end{abstract}

\maketitle

%

\section{Introduction}
\label{sec:intro}

Let $\E$ be a regular Dirichlet form and denote by $T_t$ the corresponding symmetric Markovian semigroup. For $p \in (1, \infty)$, the \emph{Sobolev--Bregman form} describes the rate of decrease of the $L^p$ norm of $T_t u$ with respect to time: we have $p \E_p(T_t u) = -\tfrac{d}{dt} \|T_t u\|_p^p$. When $p = 2$, $\E_p$ coincides with the original Dirichlet form $\E$. Sobolev--Bregman forms can be traced back to~\cite{s}, see also~\cite{cks,lps,ls,v}, although the name was introduced only recently in~\cite{bgpr} in the context of Douglas-type identities for Lévy operators. Since then, Sobolev--Bregman forms have attracted significant attention. We refer to~\cite{bgp,bgpr} for a detailed discussion and references, to~\cite{bjlp} for an application in the study of Schrödinger operators, to~\cite{bkp} for a probabilistic point of view and extensions, and to~\cite{bdl,bfr,bgpr2,kl} for related developments in the context of operators in domains.

The celebrated Beurling--Deny formula provides a decomposition of a regular Dirichlet form $\E$ into the strongly local term $\El$, the purely nonlocal part given in terms of the \emph{jumping kernel} $J$, and the killing term described by the \emph{killing measure} $k$. Our main result, Theorem~\ref{thm:main}, provides a similar characterisation of the corresponding Sobolev--Bregman form. Specifically, it identifies two functionals on $L^p$, including their domains: one defined in terms of the derivative of $T_t u$ at $t = 0$, see~\eqref{eq:pform:heat}, and another one given explicitly in terms of $\El$, $J$, and $k$, see~\eqref{eq:pform}.

For the fractional Laplace operator $(-\Delta)^s$, where $s \in (0, 1)$, this characterisation was proved in~\cite{bjlp}. An extension to similar Lévy operators is given in~\cite{bgp}. The argument used in these works requires a pointwise estimate of the kernel of $\tfrac{1}{t} T_t$ by the jumping kernel $J$, which need not hold in the general context considered in the present work.

A recent article~\cite{g} of the first named author uses a different approach and covers all pure-jump symmetric Markov semigroups. However, it relies on weak convergence of kernels of $\tfrac{1}{t} T_t$ to the jumping kernel $J$, and accordingly, it only applies to continuous functions in the domain of the Sobolev--Bregman form.

This paper provides a fully general result, relying solely on the theory of Dirichlet forms and Markovian semigroups, along with an elementary but crucial bound, given in Lemma~\ref{lem:main}. While the proof of this key lemma avoids advanced techniques, it requires careful and delicate estimates. The core idea is to use Lemma~\ref{lem:main} to compare the Sobolev--Bregman form and its approximate forms with their $L^2$ counterparts. However, integrability issues make this approximation procedure highly nontrivial.

As a sample application, in Corollary~\ref{cor:hs} we prove a Hardy--Stein identity in an equally general context. Before we state our main results, we introduce the necessary definitions. A more thorough discussion is given in Section~\ref{sec:pre}, and for a complete exposition we refer to~\cite{fot}.

\medskip

Suppose that $E$ is a locally compact, separable metric space, and $m$ is a Radon measure on $E$ with full support. We consider a regular Dirichlet form $\E$ on $E$, and we denote by $T_t$ the associated Markovian semigroup. The definition~\eqref{eq:pform:heat} of the Sobolev--Bregman form is motivated by the following relation between $\E$ and $T_t$:
\begin{align}
\label{eq:form:heat}
\begin{aligned}
 \E(u, v) & = \lim_{t \to 0^+} \frac{1}{t} \int_E (u(x) - T_t u(x)) v(x) m(dx)
\end{aligned}
\end{align}
for every $u, v$ in the domain $\D(\E)$ of $\E$. Additionally, $u \in \D(\E)$ if and only if $u \in L^2(E)$ and $\E(u, u)$ is finite. As is customary, we simply write $\E(u)$ for the quadratic form $\E(u, u)$.

The regular Dirichlet form $\E$ is given by the \emph{Beurling--Deny formula}: there is a \emph{strongly local form} $\El$, a symmetric \emph{jumping kernel} $J$ and a \emph{killing measure} $k$ such that
\begin{align}
\notag
 \E(u, v) & = \El(u, v) \\
\label{eq:form}
 & \qquad + \frac{1}{2} \iint\limits_{(E \times E) \setminus \Delta} (u(y) - u(x)) (v(y) - v(x)) J(dx, dy) \\
\notag
 & \qquad\qquad + \int_E u(x) v(x) k(dx)
\end{align}
for every $u, v \in \D(\E)$, provided that we choose quasi-continuous versions of $u, v$. Here and below $\Delta = \{(x, x) : x \in E\}$ is the diagonal in $E \times E$. Note that for every $u \in \D(\E)$ (in the quasi-continuous version) both integrals in
\begin{align}
\label{eq:form:quadratic}
 \E(u) & = \El(u) + \frac{1}{2} \iint\limits_{(E \times E) \setminus \Delta} (u(y) - u(x))^2 J(dx, dy) + \int_E (u(x))^2 k(dx)
\end{align}
are finite.

Recall that functions $u \in \D(\E)$ are only defined almost everywhere (with respect to $m$), but every $u \in \D(\E)$ is quasi-continuous after modification on a set of zero measure $m$; see Section~\ref{sec:pre}. We say that the form $\E$ is \emph{maximally defined} if every quasi-continuous function $u \in L^2(E)$ such that both integrals in~\eqref{eq:form:quadratic} are finite belongs to the domain $\D(\E)$ of the form $\E$. This appears to be a relatively mild assumption if $\E$ is a pure-jump Dirichlet form; a more detailed discussion is postponed to Section~\ref{sec:pre}.

If $p \in (1, \infty)$, the corresponding Sobolev--Bregman form (or $p$-form) $\E_p$ is defined in a similar way as in~\eqref{eq:form:heat}:
\begin{align}
\label{eq:pform:heat}
 \E_p(u) & = \lim_{t \to 0^+} \frac{1}{t} \int_E (u(x) - T_t u(x)) u\pow{p - 1}(x) m(dx) ,
\end{align}
with domain $\D(\E_p)$ consisting of those $u \in L^p(E)$ for which a finite limit exists. Here and below $s\pow{\alpha} = \abs{s}^\alpha \sign s$, and $u\pow{\alpha}(x) = (u(x))\pow{\alpha}$. Noteworthy, it is known that $\D(\E_p)$ need not be a vector space: it may fail to be closed under addition~\cite{rutkowski}.

If $u$ and $u\pow{p - 1}$ are in the domain of $\E$, then we readily have $\E_p(u) = \E(u, u\pow{p - 1})$. In this case, Beurling--Deny formula~\eqref{eq:form} implies that
\begin{align}
\notag
 \E_p(u) & = \El(u, u\pow{p - 1}) \\
\label{eq:pform:gen}
 & \qquad + \frac{1}{2} \iint\limits_{(E \times E) \setminus \Delta} (u(y) - u(x)) (u\pow{p - 1}(y) - u\pow{p - 1}(x)) J(dx, dy) \\
\notag
 & \qquad\qquad + \int_E \abs{u(x)}^p k(dx)
\end{align}
whenever $u$ is taken as the quasi-continuous version. It is thus natural to ask whether a similar identity holds for all $u \in \D(\E_p)$.

An affirmative answer to this question was known in two cases. For a class of translation-invariant Dirichlet forms $\E$ (such forms correspond to symmetric Lévy operators), this was proved in Lemma~7 in~\cite{bjlp} and Proposition~13 in~\cite{bgp}. The case of pure-jump Dirichlet forms $\E$ with an additional assumption that $u$ is continuous was given in~\cite{g}. Our main result is fully general.

\begin{theorem}[Beurling--Deny formula for Sobolev--Bregman forms]
\label{thm:main}
Let $\E$ be a regular Dirichlet form and $p \in (1, \infty)$. Then the domain $\D(\E_p)$ of the Sobolev--Bregman form $\E_p$ is characterised by
\begin{align*}
 \D(\E_p) & = \{u \in L^p(E) : u\pow{p/2} \in \D(\E)\} ,
\end{align*}
and for every $u \in \D(\E_p)$ we have
\begin{align}
\label{eq:pform:comparability}
 \frac{4 (p - 1)}{p^2} \, \E_2(u\pow{p/2}) \le \E_p(u) & \le 2 \E_2(u\pow{p/2}) .
\end{align}
Furthermore, for every $u \in \D(\E_p)$ we have the following analogue of the Beurling--Deny formula:
\begin{align}
\notag
 \E_p(u) & = \frac{4 (p - 1)}{p^2} \, \El(u\pow{p/2}) \\
\label{eq:pform}
  & \qquad + \frac{1}{2} \iint\limits_{(E \times E) \setminus \Delta} (\tilde{u}(y) - \tilde{u}(x)) (\tilde{u}\pow{p - 1}(y) - \tilde{u}\pow{p - 1}(x)) J(dx, dy) \\
\notag
  & \qquad\qquad + \int_E \abs{\tilde{u}(x)}^p k(dx) ,
\end{align}
where $\tilde{u}$ is the quasi-continuous modification of $u$.

In particular, if $u \in \D(\E_p)$, then $u$ has a quasi-continuous modification $\tilde{u}$ such that both integrals in~\eqref{eq:pform} are finite. If the Dirichlet form $\E$ is maximally defined, then the converse is true: every $u \in L^p(E)$ which has a quasi-continuous modification $\tilde{u}$ such that the two integrals in~\eqref{eq:pform} are finite, belongs to $\D(\E_p)$.
\end{theorem}

When $E = \R^n$ and the domain of $\E$ contains smooth compactly supported functions, then the strongly local part $\El$ has a more explicit description: if $u$ is sufficiently regular, we have
\begin{align*}
 \El(u) & = \int_E \sum_{i, j = 1}^n \frac{\partial u}{\partial x_i}(x) \frac{\partial u}{\partial x_j}(x) \nu_{i, j}(dx)
\end{align*}
for some locally finite measures $\nu_{i, j}$. More details are given in Section~\ref{sec:example}. The above formula for $\El$ explains the constant $4 p^{-2} (p - 1)$ in~\eqref{eq:pform}: for sufficiently regular $u$, we have $\nabla(u\pow{p - 1}) = (p - 1) |u|^{p - 2} \nabla u$ and $\nabla(u\pow{p/2}) = \tfrac{p}{2} |u|^{p/2} \nabla u$, leading to equality of the strictly local parts in~\eqref{eq:pform:gen} and~\eqref{eq:pform}, $\El(u, u\pow{p - 1}) = 4 p^{-2} (p - 1) \El(u\pow{p/2}, u\pow{p/2})$. While this calculation is valid only when $E = \R^n$, it turns out that the constant remains the same in the general case.

We remark that~\eqref{eq:pform} can be equivalently written as
\begin{align*}
 \E_p(u) & = \frac{4 (p - 1)}{p^2} \El(u\pow{p/2}) + \frac{1}{p} \iint\limits_{(E \times E) \setminus \Delta} F_p(\tilde{u}(x), \tilde{u}(y)) J(dx, dy) + \int_E \abs{\tilde{u}(x)}^p k(dx) ,
\end{align*}
where for $\xi, \eta \in \R$,
\begin{align*}
 F_p(\xi, \eta) & = \abs{\eta}^p - \abs{\xi}^p - p \xi\pow{p - 1} (\eta - \xi)
\end{align*}
is the \emph{Bregman divergence}. To prove that this is indeed equivalent to~\eqref{eq:pform}, it suffices to observe that the jumping kernel $J$ is symmetric, $F_p(\xi, \eta) \ge 0$, and $F_p(\xi, \eta) + F_p(\eta, \xi) = p (\xi - \eta) (\xi\pow{p - 1} - \eta\pow{p - 1})$. We also note that this connection with the Bregman divergence was the reason for the authors of~\cite{bgpr} to propose the name \emph{Sobolev--Bregman form}.

Comparability~\eqref{eq:pform:comparability} of $\E_p(u)$ and $\E(u\pow{p/2})$ for $u$ in the domain of the generator of $T_t$ on $L^p(E)$ is relatively straightforward. This special case of Theorem~\ref{thm:main} has appeared in various contexts in the literature, sometimes with additional assumptions $p \ge 2$ or $u \ge 0$. A one-sided bound was given as Lemma~9.9 in~\cite{s} in the study of logarithmic Sobolev inequalities. The two-sided estimate appeared as equation~(3.17) in~\cite{cks} in the proof of upper bounds for the heat kernel. The same result was given in Theorem~1 in~\cite{ls} and Theorem~3.1 in~\cite{lps} in the context of perturbation theory. The variety of applications of this special case of Theorem~\ref{thm:main} hints how useful the result in full generality can be.

We expect Theorem~\ref{thm:main} to stimulate the study of general symmetric Markovian semigroups on $L^p$ for $p \ne 2$. Sobolev--Bregman forms found applications in nonlinear nonlocal PDEs in~\cite{bgpr} and in Hardy inequalities for the fractional Laplacian in~\cite{bjlp}. In the latter reference, the authors consider the fractional Laplace operator perturbed by a particular Schrödinger potential and study when the corresponding semigroups of non-Markovian operators remain contractions on $L^p(E)$. Our result may enable similar problems to be addressed in greater generality.

As we have already noted above, the Sobolev--Bregman form describes the rate of decrease of the $L^p(E)$ norm of $T_t u$. More precisely, we have $\tfrac{d}{d t} \norm{T_t u}_p^p = -p \E_p(T_t u)$ for $t > 0$; see the proof of Theorem~3.1 in~\cite{g}. As an immediate consequence of our Theorem~\ref{thm:main}, we obtain the following explicit variant of the \emph{Hardy--Stein identity}, which generalises the main result (Theorem~1.1) of~\cite{g}.

\begin{corollary}
\label{cor:hs}
Let $\E$ be a regular Dirichlet form with the jumping kernel $J$ and the killing measure $k$, and let $p \in (1, \infty)$. For every $u \in L^p(E)$, we have
\begin{align}
\notag
 & \norm{u}_p^p - \lim_{t \to \infty} \norm{T_t u}_p^p \\
\label{eq:hs}
 & \qquad = \frac{4 (p - 1)}{p} \int_0^\infty \El((T_t u)\pow{p/2}) dt \\
\notag
 & \qquad\qquad + \frac{p}{2} \int_0^\infty \iint\limits_{(E \times E) \setminus \Delta} (T_t u(y) - T_t u(x)) ((T_t u(y))\pow{p - 1} - (T_t u(x))\pow{p - 1}) J(dx, dy) dt \\
\notag
 & \qquad\qquad\qquad + p \int_0^\infty \int_E \abs{T_t u(x)}^p k(dx) dt ,
\end{align}
where $T_t u$ is assumed to be the quasi-continuous modification.
\end{corollary}

By Theorem~\ref{thm:main}, the right-hand side of~\eqref{eq:hs} is equal to $\int_0^\infty p \E_p(T_t u) dt$ (note that $T_t u \in \D(\E_p)$ for every $t > 0$; see the proof of Theorem~3.1 in~\cite{g}). Thus, Corollary~\ref{cor:hs} indeed follows directly from Theorem~3.1 and Remark~3.2 in~\cite{g} (see also formula~(1.1) in~\cite{v}).

The above Hardy--Stein identity may find applications in the Littlewood--Paley--Stein theory for nonlocal operators, in a similar way as in~\cite{bbl,bk}. For closely related results obtained using different methods, we refer to~\cite{bb,kim,kk,kkk,lw}. We also expect applications in stochastic differential equations, as in~\cite{kkk}.

Finally, we point out that Corollary~\ref{cor:hs} resolves the following problem about Hardy--Stein identity and its applications posed by the authors of~\cite{bbl}: \emph{The results should hold in a much more general setting, but the scope of the extension is unclear at this moment} (see p.~463 therein). Our result also shows that the claim made by the authors of~\cite{lw}: \emph{It seems that such an identity depends heavily on the characterisation of Lévy processes, and may not hold for general jump processes} (see p.~424 therein) is not true.

The remaining part of the paper consists of four sections. Elementary lemmas are gathered in Section~\ref{sec:est}. In Section~\ref{sec:pre} we recall the definition of regular Dirichlet forms and their properties. Theorem~\ref{thm:main} is proved in Section~\ref{sec:proof}. In Section~\ref{sec:example} we discuss Dirichlet and Sobolev--Bregman forms on Euclidean spaces.

%
%

\section{Elementary estimates}
\label{sec:est}

Our first result in this section is a well-known auxiliary inequality. When $s, t \ge 0$, this is often called \emph{Stroock's inequality}; see the proof of Lemma~9.9 in~\cite{s}, Lemma on page~246 of~\cite{v}, or page~269 in~\cite{cks}. The general case is given, for example, as Lemma~1 in~\cite{ls}, or Lemma~2.1 in~\cite{lps}. For completeness, we present a brief proof.

\begin{lemma}
\label{lem:forms}
Let $\alpha \in (0, 2)$. Then
\begin{align}
\label{eq:forms}
 \alpha (2 - \alpha) (t - s)^2 \le (t\pow{\alpha} - s\pow{\alpha}) (t\pow{2 - \alpha} - s\pow{2 - \alpha}) \le 2 (t - s)^2
\end{align}
for all $s, t \in \R$.
\end{lemma}

\begin{proof}
By symmetry, with no loss of generality, we may assume that $s \le t$. Let us denote by $I$ the middle expression in~\eqref{eq:forms}:
\begin{align*}
 I & = (t\pow{\alpha} - s\pow{\alpha}) (t\pow{2 - \alpha} - s\pow{2 - \alpha}) .
\end{align*}
If $s t \ge 0$, then, by the AM-GM inequality,
\begin{align*}
 \frac{s\pow{\alpha} t\pow{2 - \alpha} + t\pow{\alpha} s\pow{2 - \alpha}}{2} & = \frac{\abs{s}^\alpha \abs{t}^{2 - \alpha} + \abs{t}^\alpha \abs{s}^{2 - \alpha}}{2} \ge \abs{s} \abs{t} = s t ,
\end{align*}
and so
\begin{align*}
 I  & = s^2 + t^2 - (s\pow{\alpha} t\pow{2 - \alpha} + t\pow{\alpha} s\pow{2 - \alpha}) \\
 & \le s^2 + t^2 - 2 s t = (t - s)^2 .
\end{align*}
If $s t < 0$, then $t\pow{\alpha} + s\pow{\alpha}$ and $t\pow{2 - \alpha} + s\pow{2 - \alpha}$ have equal sign, and hence
\begin{align*}
 I & \le I + (t\pow{\alpha} + s\pow{\alpha}) (t\pow{2 - \alpha} + s\pow{2 - \alpha}) \\
 & = 2 s^2 + 2 t^2 \le 2 (s^2 + t^2 - 2 s t) = 2 (t - s)^2 .
\end{align*}
This completes the proof of the upper bound. For the lower bound, observe that
\begin{align*}
 I & = \alpha (2 - \alpha) \int_s^t \int_s^t \abs{x}^{\alpha - 1} \abs{y}^{1 - \alpha} dy dx \\
 & = \alpha (2 - \alpha) \int_s^t \int_s^t \frac{\abs{x}^{\alpha - 1} \abs{y}^{1 - \alpha} + \abs{y}^{\alpha - 1} \abs{x}^{1 - \alpha}}{2} \, dy dx .
\end{align*}
Using the AM--GM inequality again, we see that the integrand on the right-hand side is at least $1$, and hence
\begin{align*}
 I & \ge \alpha (2 - \alpha) (t - s)^2 ,
\end{align*}
as desired.
\end{proof}

The following lemma is our key technical result.

\begin{figure}
\centering
\begin{tikzpicture}[scale=0.72]
\footnotesize
\draw[dotted] (1,5) -- (1,0) node[below right] {$1$\vphantom{$n^4$}} -- (1,-5);
\draw[dotted] (-1,5) -- (-1,0) node[below left] {$-1$\vphantom{$n^4$}} -- (-1,-5);
\draw[dotted] (2,5) -- (2,0) node[below right] {$n$\vphantom{$n^4$}} -- (2,-2.5);
\draw[dotted] (-2,5) -- (-2,0) node[below left] {$-n$\vphantom{$n^4$}} -- (-2,-5);
\draw[dotted] (4,5) -- (4,0) node[below right] {$n^3$\vphantom{$n^4$}} -- (4,-5);
\draw[dotted] (-4,5) -- (-4,0) node[below left] {$-n^3$\vphantom{$n^4$}} -- (-4,-5);
\draw[dotted] (5,5) -- (5,0) node[below right] {$n^4$\vphantom{$n^4$}} -- (5,-5);
\draw[dotted] (-5,5) -- (-5,0) node[below left] {$-n^4$\vphantom{$n^4$}} -- (-5,-5);
\draw[dotted] (-6.5,2) -- (0,2) node[above left] {$1$\vphantom{$n^4$}} -- (6.5,2);
\draw[dotted] (-6.5,-2) -- (0,-2) node[below left] {$-1$\vphantom{$n^4$}} -- (6.5,-2);
\draw[very thick,teal] (-6.5,-4.4) -- (-5,-4.4) .. controls (-3,-3.8) and (-2,-3) .. (-1,-2) -- (0,0) -- (1,2) .. controls (2,3) and (3,3.8) .. (5,4.4) -- (6.5,4.4);
\node[teal,left] at (6.5,3.6) {$\ph_{\alpha, n}$};
\draw[very thick,dashed] (-6.5,-4.8) .. controls (-6,-4.7) and (-5.5,-4.6) .. (-5,-4.4) .. controls (-3,-3.8) and (-2,-3) .. (-1,-2) .. controls (-0.5,-1.5) and (0,-1) .. (0,0) .. controls (0,1) and (0.5,1.5) .. (1,2) .. controls (2,3) and (3,3.8) .. (5,4.4) .. controls (5.5,4.6) and (6,4.7) .. (6.5,4.8);
\node[left] at (6.5,5.4) {$\ph_\alpha$};
\draw[->] (-6.5,0) -- (6.5,0) node[below] {$t$};
\draw[->] (0,-5) -- (0,5) node[left] {$y$};
\end{tikzpicture}
\begin{tikzpicture}[scale=0.72]
\footnotesize
\draw[dotted] (1,6.5) -- (1,0) node[below right] {$1$\vphantom{$n^4$}} -- (1,-6.5);
\draw[dotted] (-1,6.5) -- (-1,0) node[below left] {$-1$\vphantom{$n^4$}} -- (-1,-6.5);
\draw[dotted] (2,6.5) -- (2,0) node[below right] {$n$\vphantom{$n^4$}} -- (2,-6.5);
\draw[dotted] (-2,6.5) -- (-2,0) node[below left] {$-n$\vphantom{$n^4$}} -- (-2,-6.5);
\draw[dotted] (4,6.5) -- (4,0) node[below right] {$n^3$\vphantom{$n^4$}} -- (4,-6.5);
\draw[dotted] (-4,6.5) -- (-4,0) node[below left] {$-n^3$\vphantom{$n^4$}} -- (-4,-6.5);
\draw[dotted] (5,6.5) -- (5,0) node[below right] {$n^4$\vphantom{$n^4$}} -- (5,-6.5);
\draw[dotted] (-5,6.5) -- (-5,0) node[below left] {$-n^4$\vphantom{$n^4$}} -- (-5,-6.5);
\draw[dotted] (-6.5,0.7) -- (0,0.7) node[above left] {$1$\vphantom{$n^4$}} -- (6.5,0.7);
\draw[dotted] (-6.5,-0.7) -- (0,-0.7) node[below left] {$-1$\vphantom{$n^4$}} -- (6.5,-0.7);
\draw[very thick,teal] (-6.5,-5.5) -- (-5,-5.5) .. controls (-3.5,-3) and (-3,-2.7) .. (-1,-0.7) -- (0,0) -- (1,0.7) .. controls (3,2.7) and (3.5,3) .. (5,5.5) -- (6.5,5.5);
\node[teal,left] at (6.6,5) {$\ph_{2 - \alpha, n}$};
\draw[very thick,dashed] (-5.6,-6.5) .. controls (-5.4,-6.1) and (-5.2,-5.8) .. (-5,-5.5) .. controls (-3.5,-3) and (-3,-2.7) .. (-1,-0.7) .. controls (-0.6,-0.3) and (-0.3,0) .. (0,0) .. controls (0.3,0) and (0.6,0.3) .. (1,0.7) .. controls (3,2.7) and (3.5,3) .. (5,5.5) .. controls (5.2,5.8) and (5.4,6.1) .. (5.6,6.5);
\node[left] at (5.3,6.1) {$\ph_{2 - \alpha}$};
\draw[->] (-6.5,0) -- (6.5,0) node[below] {$t$};
\draw[->] (0,-6.5) -- (0,6.5) node[left] {$y$};
\end{tikzpicture}
\begin{tikzpicture}[scale=0.72]
\footnotesize
\draw[dotted] (1,4) -- (1,0) node[below right] {$1$\vphantom{$n^4$}} -- (1,-4);
\draw[dotted] (-1,4) -- (-1,0) node[below left] {$-1$\vphantom{$n^4$}} -- (-1,-4);
\draw[dotted] (2,4) -- (2,0) node[below right] {$n$\vphantom{$n^4$}} -- (2,-4);
\draw[dotted] (-2,4) -- (-2,0) node[below left] {$-n$\vphantom{$n^4$}} -- (-2,-4);
\draw[dotted] (4,4) -- (4,0) node[below right] {$n^3$\vphantom{$n^4$}} -- (4,-4);
\draw[dotted] (-4,4) -- (-4,0) node[below left] {$-n^3$\vphantom{$n^4$}} -- (-4,-4);
\draw[dotted] (5,4) -- (5,0) node[below right] {$n^4$\vphantom{$n^4$}} -- (5,-4);
\draw[dotted] (-5,4) -- (-5,0) node[below left] {$-n^4$\vphantom{$n^4$}} -- (-5,-4);
\draw[dotted] (-6.5,1) -- (0,1) node[above left] {$1$\vphantom{$n^4$}} -- (6.5,1);
\draw[dotted] (-6.5,-1) -- (0,-1) node[below left] {$-1$\vphantom{$n^4$}} -- (6.5,-1);
\draw[dotted] (-6.5,2) -- (0,2) node[above left] {$n$\vphantom{$n^4$}} -- (6.5,2);
\draw[dotted] (-6.5,-2) -- (0,-2) node[below left] {$-n$\vphantom{$n^4$}} -- (6.5,-2);
\draw[very thick] (-6,-4) -- (-4,-2) -- (-2,-2) -- (2,2) -- (4,2) -- (6,4);
\node[left] at (6.5,3.5) {$\psi_n$};
\draw[->] (-6.5,0) -- (6.5,0) node[below] {$t$};
\draw[->] (0,-4) -- (0,4) node[left] {$y$};
\end{tikzpicture}
\caption{Functions defined in Lemma~\ref{lem:main} (not to scale).}
\label{fig:functions}
\end{figure}
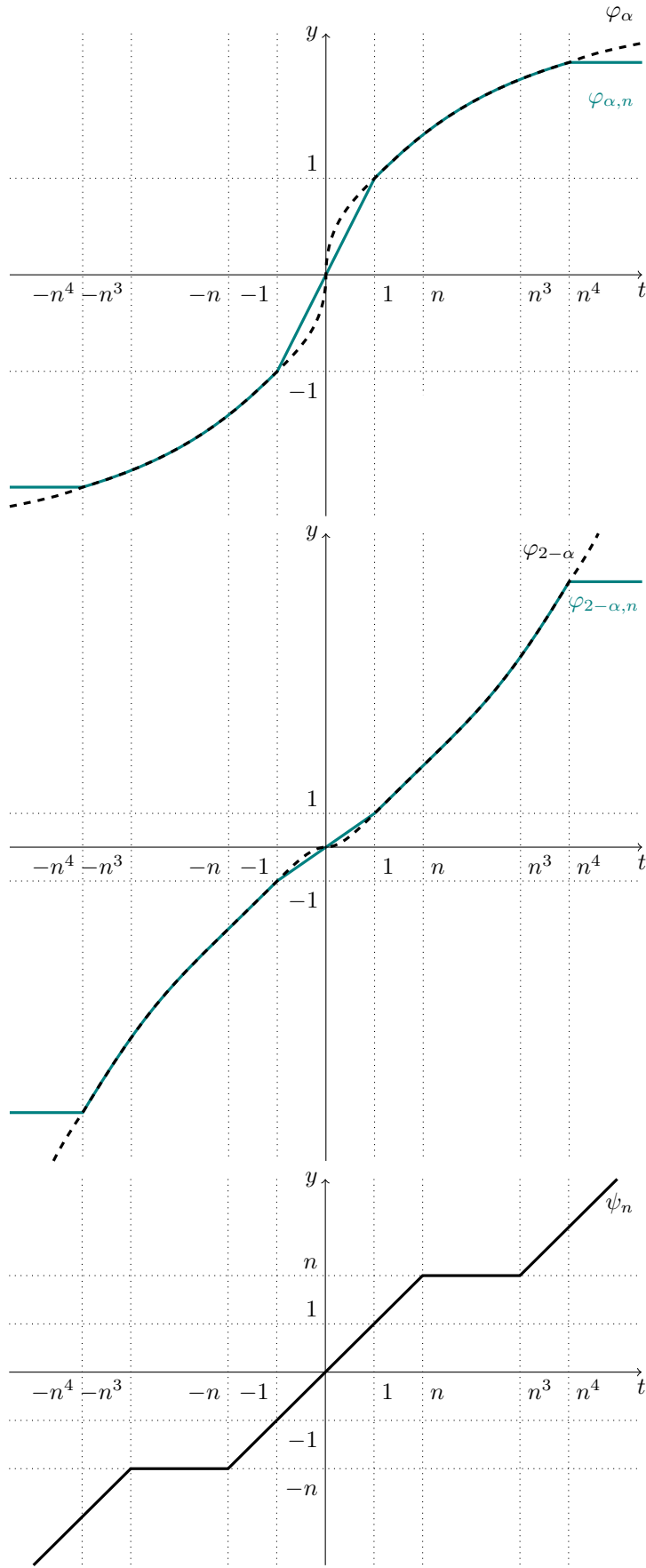

\begin{lemma}
\label{lem:main}
For $\alpha \in (0, 2)$ and a number $n \ge 2$, denote (see Figure~\ref{fig:functions})
\begin{align*}
 \ph_\alpha(s) & = s\pow{\alpha} , \\
 \ph_{\alpha, n}(s) & = \begin{cases}
  s \hspace*{8em} & \text{if $\abs{s} < 1$}, \\
  s\pow{\alpha} & \text{if $1 \le \abs{s} < n^4$}, \\
  n^{4 \alpha} \sign s & \text{if $n^4 \le \abs{s}$},
 \end{cases} \\
 \psi_n(s) & = \begin{cases}
  s \hspace*{8em} & \text{if $\abs{s} < n$}, \\
  n \sign s & \text{if $n \le \abs{s} < n^3$}, \\
  s - (n^3 - n) \sign s & \text{if $n^3 \le \abs{s}$}.
 \end{cases}
\end{align*}
Then,
\begin{align}
\label{eq:main}
 \begin{aligned}
 & \Bigl| (\ph_{\alpha, n}(t) - \ph_{\alpha, n}(s)) (\ph_{2 - \alpha, n}(t) - \ph_{2 - \alpha, n}(s)) \\
 & \qquad\qquad - (\ph_\alpha(t) - \ph_\alpha(s))(\ph_{2 - \alpha}(t) - \ph_{2 - \alpha}(s)) \Bigr| \\
 & \qquad\qquad\qquad\qquad \le 8 n^{-\min\{\alpha, 2 - \alpha\}} (t - s)^2 + 180 (\psi_n(t) - \psi_n(s))^2
 \end{aligned}
\end{align}
for all $n \ge 2$ and all $s, t \in \R$.
\end{lemma}

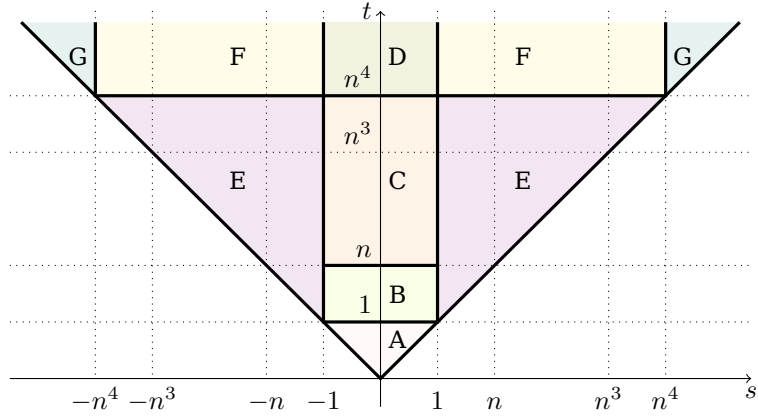
\begin{figure}
\centering
\begin{tikzpicture}[scale=0.75]
\footnotesize
\fill[fill=pink!10!white] (-1,1) -- (0,0) -- (1,1);
\node at (0.3,0.7) {A};
\fill[fill=lime!10!white] (-1,1) -- (1,1) -- (1,2) -- (-1,2);
\node at (0.3,1.5) {B};
\fill[fill=orange!10!white] (-1,2) -- (1,2) -- (1,5) -- (-1,5);
\node at (0.3,3.5) {C};
\fill[fill=olive!10!white] (-1,5) -- (1,5) -- (1,6.3) -- (-1,6.3);
\node at (0.3,5.7) {D};
\fill[fill=violet!10!white] (-1,1) -- (-1,5) -- (-5,5);
\fill[fill=violet!10!white] (1,1) -- (1,5) -- (5,5);
\node at (-2.5,3.5) {E};
\node at (2.5,3.5) {E};
\fill[fill=yellow!10!white] (-1,5) -- (-1,6.3) -- (-5,6.3) -- (-5,5);
\fill[fill=yellow!10!white] (1,5) -- (1,6.3) -- (5,6.3) -- (5,5);
\node at (-2.5,5.7) {F};
\node at (2.5,5.7) {F};
\fill[fill=teal!10!white] (-5,5) -- (-5,6.3) -- (-6.3,6.3);
\fill[fill=teal!10!white] (5,5) -- (5,6.3) -- (6.3,6.3);
\node at (-5.3,5.7) {G};
\node at (5.3,5.7) {G};
\draw[very thick] (-6.3,6.3) -- (0,0) -- (6.3,6.3);
\draw[very thick] (-5,6.3) -- (-5,5) -- (5,5) -- (5,6.3);
\draw[very thick] (-1,6.3) -- (-1,1) -- (1,1) -- (1,6.3);
\draw[very thick] (-1,2) -- (1,2);
\draw[dotted] (1,6.5) -- (1,0) node[below] {$1$\vphantom{$n^4$}};
\draw[dotted] (-1,6.5) -- (-1,0) node[below] {$-1$\vphantom{$n^4$}};
\draw[dotted] (2,6.5) -- (2,0) node[below] {$n$\vphantom{$n^4$}};
\draw[dotted] (-2,6.5) -- (-2,0) node[below] {$-n$\vphantom{$n^4$}};
\draw[dotted] (4,6.5) -- (4,0) node[below] {$n^3$\vphantom{$n^4$}};
\draw[dotted] (-4,6.5) -- (-4,0) node[below] {$-n^3$\vphantom{$n^4$}};
\draw[dotted] (5,6.5) -- (5,0) node[below] {$n^4$\vphantom{$n^4$}};
\draw[dotted] (-5,6.5) -- (-5,0) node[below] {$-n^4$\vphantom{$n^4$}};
\draw[dotted] (-6.5,1) -- (0,1) node[above left] {$1$\vphantom{$n^4$}} -- (6.5,1);
\draw[dotted] (-6.5,2) -- (0,2) node[above left] {$n$\vphantom{$n^4$}} -- (6.5,2);
\draw[dotted] (-6.5,4) -- (0,4) node[above left] {$n^3$\vphantom{$n^4$}} -- (6.5,4);
\draw[dotted] (-6.5,5) -- (0,5) node[above left] {$n^4$\vphantom{$n^4$}} -- (6.5,5);
\draw[->] (-6.5,0) -- (6.5,0) node[below] {$s$};
\draw[->] (0,-0.5) -- (0,6.5) node[left] {$t$};
\end{tikzpicture}
\caption{Regions considered in the proof of Lemma~\ref{lem:main} (not to scale).}
\label{fig:diagram}
\end{figure}

\begin{proof}
We divide the argument into six steps.

\emph{Step 1.} We begin with elementary simplifications. By symmetry, with no loss of generality we may assume that $\abs{s} \le \abs{t}$, and since $\ph_\alpha$, $\ph_{\alpha, n}$, and $\psi_n$ are odd functions, we may additionally assume that $t \ge 0$. Thus, it is sufficient to consider $s, t \in \R$ such that $-t \le s \le t$ (see Figure~\ref{fig:diagram}). Furthermore, the statement of the lemma does not change when $\alpha$ is replaced by $2 - \alpha$, and therefore we may restrict our attention to $\alpha \in (0, 1]$. We maintain these assumptions throughout the proof.

To simplify the notation, we denote
\begin{align*}
 \Phi_\alpha(s, t) & = (\ph_\alpha(t) - \ph_\alpha(s))(\ph_{2 - \alpha}(t) - \ph_{2 - \alpha}(s)) , \\
 \Phi_{\alpha, n}(s, t) & = (\ph_{\alpha, n}(t) - \ph_{\alpha, n}(s)) (\ph_{2 - \alpha, n}(t) - \ph_{2 - \alpha, n}(s)) , \\
 \Psi_n(s, t) & = (\psi_n(t) - \psi_n(s))^2 .
\end{align*}
Thus, the desired inequality~\eqref{eq:main} can be written as
\begin{align}
\label{eq:main:alt}
 \abs{\Phi_{\alpha, n}(s, t) - \Phi_\alpha(s, t)} & \le \eps_{\alpha, n} (t - s)^2 + c \Psi_n(s, t) ,
\end{align}
where $c = 180$ and $\eps_{\alpha, n} = 8 n^{-\alpha}$. We split the region $\abs{s} \le t$ into a number of subregions, as shown in Figure~\ref{fig:diagram}.

We first gather the necessary estimates of $\Phi_\alpha$ and $\Phi_{\alpha, n}$, then we estimate $\Psi_n$, and only then we return to the actual proof of~\eqref{eq:main:alt}.

\emph{Step 2.} We have the following immediate estimates of $\Phi_\alpha$ and $\Phi_{\alpha, n}$. By definition,
\begin{align}
\label{eq:main:3}
 \Phi_{\alpha, n}(s, t) & = \Phi_\alpha(s, t) && \text{if } \underbrace{1 \le \abs{s} \le t \le n^4}_{\text{region E}} ,
\end{align}
and
\begin{align}
\label{eq:main:4}
 \Phi_{\alpha, n}(s, t) & = (s - t)^2 && \text{if } \underbrace{\abs{s} \le t \le 1}_{\text{region A}} .
\end{align}
Lemma~\ref{lem:forms} implies that
\begin{align}
\label{eq:main:1}
 0 \le \Phi_\alpha(s, t) & \le 2 (t - s)^2 && \text{if } \underbrace{\abs{s} \le t}_{\text{all regions}} .
\end{align}

\emph{Step 3.}
We turn to an estimate for $\Phi_{\alpha, n}(s, t)$ similar to~\eqref{eq:main:1} in all regions but~A and~B. If $1 \le \abs{s} \le t$, then
\begin{align*}
 0 & \le \ph_{\alpha, n}(t) - \ph_{\alpha, n}(s) \le \ph_\alpha(t) - \ph_\alpha(s) , \\
 0 & \le \ph_{2 - \alpha, n}(t) - \ph_{2 - \alpha, n}(s) \le \ph_{2 - \alpha}(t) - \ph_{2 - \alpha}(s) ,
\end{align*}
and hence
\begin{align*}
 0 & \le \Phi_{\alpha, n}(s, t) \le \Phi_\alpha(s, t) .
\end{align*}
Combining this with~\eqref{eq:main:1}, we arrive at
\begin{align}
\label{eq:main:2}
 0 & \le \Phi_{\alpha, n}(s, t) \le 2 (t - s)^2 && \text{if } \underbrace{1 \le \abs{s} \le t}_{\text{regions E, F, G}} .
\end{align}
We now show that a similar estimate is valid also when $\abs{s} \le 1$ and $n \le t$. In this case we have
\begin{align*}
 0 & \le \ph_{\alpha, n}(t) - \ph_{\alpha, n}(s) \le \ph_{\alpha, n}(t) - \ph_{\alpha, n}(-1) , \\
 0 & \le \ph_{2 - \alpha, n}(t) - \ph_{2 - \alpha, n}(s) \le \ph_{2 - \alpha, n}(t) - \ph_{2 - \alpha, n}(-1) ,
\end{align*}
and hence
\begin{align*}
 0 & \le \Phi_{\alpha, n}(s, t) \le \Phi_{\alpha, n}(-1, t) .
\end{align*}
By~\eqref{eq:main:2} applied with $s = -1$, we find that
\begin{align*}
 \Phi_{\alpha, n}(-1, t) & \le 2 (t + 1)^2 .
\end{align*}
Finally, since $t \ge n$ and $n \ge 2$, we have
\begin{align*}
 t + 1 & \le 3 t - 2 n + 1 \\
 & \le 3 (t - 1) \\
 & \le 3 (t - s) .
\end{align*}
Together the above estimates lead to
\begin{align*}
 \Phi_{\alpha, n}(s, t) & \le \Phi_{\alpha, n}(-1, t) \\
 & \le 2 (t + 1)^2 \\
 & \le 18 (t - s)^2 .
\end{align*}
Thus, we have the following analogue of~\eqref{eq:main:2}:
\begin{align}
\label{eq:main:6}
 0 & \le \Phi_{\alpha, n}(s, t) \le 18 (t - s)^2 && \text{if } \underbrace{\abs{s} \le 1 \le n \le t}_{\text{regions C, D}} .
\end{align}

\emph{Step 4.} A more refined estimate of $\Phi_{\alpha, n}(s, t) - \Phi_\alpha(s, t)$ is needed in regions~B and~C, when $\abs{s} \le 1 \le t \le n^4$. Observe that
\begin{align*}
 \Phi_{\alpha, n}(s, t) - \Phi_\alpha(s, t) & = (\ph_{\alpha, n}(t) - \ph_{\alpha, n}(s)) (\ph_{2 - \alpha, n}(t) - \ph_{2 - \alpha, n}(s)) \\
 & \qquad\qquad - (\ph_\alpha(t) - \ph_\alpha(s)) (\ph_{2 - \alpha}(t) - \ph_{2 - \alpha}(s)) \\
 & = (t^\alpha - s) (t^{2 - \alpha} - s) - (t^\alpha - s\pow{\alpha}) (t^{2 - \alpha} - s\pow{2 - \alpha}) \\
 & = (t^2 - t^\alpha s - t^{2 - \alpha} s + s^2) - (t^2 - t^\alpha s\pow{2 - \alpha} - t^{2 - \alpha} s\pow{\alpha} + s^2) \\
 & = t^\alpha s\pow{2 - \alpha} + t^{2 - \alpha} s\pow{\alpha} - t^\alpha s - t^{2 - \alpha} s \\
 & = s\pow{\alpha} t^\alpha (1 - \abs{s}^{1 - \alpha}) (t^{2 - 2 \alpha} - \abs{s}^{1 - \alpha}) .
\end{align*}
Therefore,
\begin{align}
\label{eq:main:10}
 \abs{\Phi_{\alpha, n}(s, t) - \Phi_\alpha(s, t)} & = \abs{s}^\alpha t^\alpha \bigl|1 - \abs{s}^{1 - \alpha}\bigr| \bigl|t^{2 - 2 \alpha} - \abs{s}^{1 - \alpha}\bigr|  && \text{if } \underbrace{\abs{s} \le t}_{\text{all regions}} .
\end{align}
Suppose that $\abs{s} \le 1$ and $n \le t$. Using~\eqref{eq:main:10}, we find that
\begin{align*}
 \abs{\Phi_{\alpha, n}(s, t) - \Phi_\alpha(s, t)} & \le t^\alpha (t^{2 - 2 \alpha} + 1) \le 2 t^{2 - \alpha} \le 2 n^{-\alpha} t^2 .
\end{align*}
Since $t \le 2 (t - 1) \le 2 (t - s)$, we arrive at
\begin{align}
\label{eq:main:8}
 \abs{\Phi_{\alpha, n}(s, t) - \Phi_\alpha(s, t)} & \le 8 n^{-\alpha} (t - s)^2 && \text{if } \underbrace{\abs{s} \le 1 \le n \le t \le n^4}_{\text{region C}} .
\end{align}
On the other hand, if $\abs{s} \le 1 \le t \le n$, then, again by~\eqref{eq:main:10},
\begin{align*}
 \abs{\Phi_{\alpha, n}(s, t) - \Phi_\alpha(s, t)} & \le t^\alpha (1 - \abs{s}^{1 - \alpha}) (t^{2 - 2 \alpha} - \abs{s}^{1 - \alpha}) \\
 & = t^\alpha (1 - \abs{s}^{1 - \alpha}) ((t^{2 - 2 \alpha} - 1) + (1 - \abs{s}^{1 - \alpha})) .
\end{align*}
We combine this estimate with $1 - \abs{s}^{1 - \alpha} \le 1 - s$ and with
\begin{align*}
 t^{2 - 2 \alpha} - 1 & = (t^{1 - \alpha} + 1) (t^{1 - \alpha} - 1) \le 2 t^{1 - \alpha} (t - 1) ,
\end{align*}
to find that
\begin{align*}
 \abs{\Phi_{\alpha, n}(s, t) - \Phi_\alpha(s, t)} & \le t^\alpha (1 - s) (2 t^{1 - \alpha} (t - 1) + (1 - s)) \\
 & \le t^\alpha (1 - s) (2 t^{1 - \alpha} (t - 1) + 2 t^{1 - \alpha} (1 - s)) \\
 & = 2 t (1 - s) (t - s) .
\end{align*}
Finally, $t (1 - s) = t - s t \le t - s$ if $s \ge 0$, and $t (1 - s) \le 2 t \le 2 (t - s)$ if $s \le 0$. Thus,
\begin{align}
\label{eq:main:9}
 \abs{\Phi_{\alpha, n}(s, t) - \Phi_\alpha(s, t)} & \le 4 (t - s)^2 && \text{if } \underbrace{\abs{s} \le 1 \le t \le n}_{\text{region B}} .
\end{align}

\emph{Step 5.} We turn to the estimates of $\Psi_n$. By definition,
\begin{align}
\label{eq:main:5}
 \Psi_n(s, t) & = (t - s)^2 && \text{if } \underbrace{\abs{s} \le t \le n}_{\text{includes regions A, B}} \text{ or } \underbrace{n^3 \le \abs{s} \le t}_{\text{includes region G}} .
\end{align}
Suppose that $\abs{s} \le n^3$ and $n^4 \le t$. Then
\begin{align*}
 \psi_n(t) - \psi_n(s) & \ge \psi_n(t) - n = t - n^3 .
\end{align*}
Since $t \ge n^4 \ge 2 n^3$, we have
\begin{align*}
 3 (\psi_n(t) - \psi_n(s)) & \ge 3 t - 3 n^3 \\
 & \ge t + n^3 \\
 & \ge t - s .
\end{align*}
It follows that
\begin{align*}
 9 \Psi_n(s, t) & = 9 (\psi_n(t) - \psi_n(s))^2 \ge (t - s)^2 .
\end{align*}
This estimate partially covers regions~D and~F, and the remaining part of these regions is included in~\eqref{eq:main:5}. Thus,
\begin{align}
\label{eq:main:7}
 9 \Psi_n(s, t) & \ge (t - s)^2 && \text{if } \underbrace{\abs{s} \le n^4 \le t}_{\text{regions D, F}} .
\end{align}

\emph{Step 6.} With the above bounds at hand, we are ready to prove~\eqref{eq:main:alt}. We consider the following seven cases, which correspond to regions shown in Figure~\ref{fig:diagram}.
\begin{itemize}
\item Region A: If $\abs{s} \le t \le 1$, then, by~\eqref{eq:main:4}, \eqref{eq:main:1} and~\eqref{eq:main:5},
\begin{align*}
 \abs{\Phi_{\alpha, n}(s, t) - \Phi_\alpha(s, t)} & \le \Phi_{\alpha, n}(s, t) + \Phi_\alpha(s, t) \\
 & \le 3 (t - s)^2 \\
 & = 3 \Psi_n(s, t) .
\end{align*}
\item Region B: If $\abs{s} \le 1 \le t \le n$, then, by~\eqref{eq:main:9} and~\eqref{eq:main:5}, we have
\begin{align*}
 \abs{\Phi_{\alpha, n}(s, t) - \Phi_\alpha(s, t)} & \le 4 (t - s)^2 \\
 & = 4 \Psi_n(s, t) .
\end{align*}
\item Region C: If $\abs{s} \le 1 \le n \le t \le n^4$, then, by~\eqref{eq:main:8}, we have
\begin{align*}
 \abs{\Phi_{\alpha, n}(s, t) - \Phi_\alpha(s, t)} & \le 8 n^{-\alpha} (t - s)^2 .
\end{align*}
\item Region D: If $\abs{s} \le 1 \le n^4 \le t$, then, by~\eqref{eq:main:1}, \eqref{eq:main:6} and~\eqref{eq:main:7} we have
\begin{align*}
 \abs{\Phi_{\alpha, n}(s, t) - \Phi_\alpha(s, t)} & \le \Phi_{\alpha, n}(s, t) + \Phi_\alpha(s, t) \\
 & \le 20 (t - s)^2 \\
 & \le 180 \Psi_n(s, t) .
\end{align*}
\item Region E: If $1 \le \abs{s} \le t \le n^4$, then, by~\eqref{eq:main:3} we have
\begin{align*}
 \abs{\Phi_{\alpha, n}(s, t) - \Phi_\alpha(s, t)} & = 0 .
\end{align*}
\item Region F: If $1 \le \abs{s} \le n^4 \le t$, then, by~\eqref{eq:main:1}, \eqref{eq:main:2} and~\eqref{eq:main:7} we have
\begin{align*}
 \abs{\Phi_{\alpha, n}(s, t) - \Phi_\alpha(s, t)} & \le \Phi_{\alpha, n}(s, t) + \Phi_\alpha(s, t) \\
 & \le 4 (t - s)^2 \\
 & \le 36 \Psi_n(s, t) .
\end{align*}
\item Region G: If $n^4 \le \abs{s} \le t$, then, by~\eqref{eq:main:1}, \eqref{eq:main:2} and~\eqref{eq:main:5} we have
\begin{align*}
 \abs{\Phi_{\alpha, n}(s, t) - \Phi_\alpha(s, t)} & \le \Phi_{\alpha, n}(s, t) + \Phi_\alpha(s, t) \\
 & \le 4 (t - s)^2 \\
 & = 4 \Psi_n(s, t) .
\end{align*}
\end{itemize}
The proof is complete.
\end{proof}

%
%

\section{Dirichlet forms and the corresponding Sobolev--Bregman forms}
\label{sec:pre}

With Lemma~\ref{lem:main} at hand, we are ready to prove our main result. First, however, we recall the definitions of the Dirichlet form $\E$ and the corresponding Sobolev--Bregman forms $\E_p$, and their basic properties. This expands on the brief introduction in Section~\ref{sec:intro}, reiterating key points for the reader's convenience.

Recall that $E$ is a locally compact, separable metric space, and $m$ is a reference measure: a Radon measure on $E$ with full support. By `equal almost everywhere', we mean equality up to a set of zero measure $m$, and $L^p(E)$ denotes the space of (real) Borel functions $u$ with finite norm $\norm{u}_p = (\int_E \abs{u(x)}^p m(dx))^{1/p}$, where functions equal almost everywhere have been identified.

For the remainder of the paper, we assume that $\E$ is a regular Dirichlet form with domain $\D(\E) \sub L^2(E)$. Recall that a symmetric bilinear form $\E$ is a Dirichlet form if it is closed and Markovian, and it is a regular Dirichlet form if additionally the set of continuous, compactly supported functions in $\D(\E)$ forms a core for $\E$. Recall also that a symmetric bilinear form $\E$ is Markovian if it has the following property: if $u \in \D(\E)$ and $v$ is a \emph{normal contraction} of $u$ (that is, $\abs{v(x)} \le \abs{u(x)}$ and $\abs{v(x) - v(y)} \le \abs{u(x) - u(y)}$ for all $x, y \in E$), then $v \in \D(\E)$ and $\E(v, v) \le \E(u, u)$. We refer to Section~1.1 in~\cite{fot} for a detailed discussion.

While functions in the domain $\D(\E)$ need not be continuous, there is a weaker notion of \emph{quasi-continuity} associated to the regular Dirichlet form $\E$; see Section~2.1 in~\cite{fot}. Throughout this paper, whenever we write \emph{quasi-continuous}, we mean a property called \emph{quasi-continuity in the restricted sense} in~\cite{fot}, that is, quasi-continuity on the one-point compactification of $E$. It is known that every $u \in \D(\E)$ is equal almost everywhere to a quasi-continuous function $\tilde{u}$, called the \emph{quasi-continuous modification} of $u$ (Theorem~2.1.3 in~\cite{fot}).

The Beurling--Deny formula~\eqref{eq:form} holds for every $u, v \in \D(\E)$, provided that we choose quasi-continuous modifications. That is, \eqref{eq:form} should be formally written as
\begin{align}
\notag
 \E(u, v) & = \El(u, v) \\
\label{eq:form:quasi}
 & \qquad + \frac{1}{2} \iint\limits_{(E \times E) \setminus \Delta} (\tilde{u}(y) - \tilde{u}(x)) (\tilde{v}(y) - \tilde{v}(x)) J(dx, dy) \\
\notag
 & \qquad\qquad + \int_E \tilde{u}(x) \tilde{v}(x) k(dx) ,
\end{align}
where $u, v \in \D(\E)$ (Lemma~4.5.4 and Theorem~4.5.2 in~\cite{fot}). The Dirichlet form $\E$ is said to be \emph{pure-jump} if $\El(u, v) = 0$. Note that our definition of the jumping measure includes a constant $\tfrac{1}{2}$ in~\eqref{eq:form:quasi}; that is, we write $\tfrac{1}{2} J(dx, dy)$ for what is denoted in~\cite{fot} by $J(dx, dy)$. This is motivated by the probabilistic interpretation of the jumping measure (see formula~(5.3.6) in~\cite{fot}), and it agrees with the notation used by various other authors.

Every $u \in \D(\E)$ has a quasi-continuous modification $\tilde{u}$ such that both integrals in
\begin{align}
\label{eq:form:quasi:quadratic}
 \E(u) & = \El(u) + \frac{1}{2} \iint\limits_{(E \times E) \setminus \Delta} (\tilde{u}(y) - \tilde{u}(x))^2 J(dx, dy) + \int_E (\tilde{u}(x))^2 k(dx)
\end{align}
are finite. We say that the form $\E$ is \emph{maximally defined} if the converse is true: every $u \in L^2(E)$ which has a quasi-continuous modification $\tilde{u}$ such that the integrals on the right-hand side of~\eqref{eq:form:quasi:quadratic} are finite, belongs to $\D(\E)$.

Many commonly used pure-jump Dirichlet forms are maximally defined; see~\cite{su}. In fact, we are not aware of any example of a regular pure-jump Dirichlet form which is not maximally defined. On the other hand, if the strongly local part $\El$ is nonvanishing, $\E$ is typically not maximally defined. \emph{Reflected Dirichlet forms} and \emph{Silverstein extensions} (see~\cite{chen,kuwae}) seem to be closely related concepts.

The strongly local part $\El$ is a symmetric Markovian form defined on the domain $\D(\E)$ of the original Dirichlet form, and $\El(u, v) = 0$ whenever $v$ is constant on a neighbourhood of the support of $u$ (formula~(4.5.14) in~\cite{fot}). If $\El$ is not identically zero, we define $\D(\El)$ to be equal to $\D(\E)$, even if $\El$ can be extended to a larger class of functions. If, however, $\El(u, v) = 0$ for every $u, v \in \D(\E)$, then we extend this equality to all $u, v \in L^2(E)$, and we write $\D(\El) = L^2(E)$. In this case $\E$ is said to be a \emph{pure-jump} (or \emph{purely nonlocal}) Dirichlet form. 

No explicit description of $\El$ is available in the general setting (see, however, Section~\ref{sec:example} for the discussion of the Euclidean case). Nevertheless, in many aspects the form $\El$ is well-understood. In particular, LeJan's formulae state that whenever $u \in \D(\El)$, there is a finite measure $\mul_u$ on $E$, called the \emph{energy measure} of $u$, such that
\begin{align}
\label{eq:form:local}
 \El(\ph(u), \psi(u)) & = \frac{1}{2} \int_E \ph'(u(x)) \psi'(u(x)) \mul_u(dx)
\end{align}
for all Lipschitz functions $\ph, \psi$ on $\R$ such that $\ph(0) = \psi(0) = 0$ (see Theorem~3.2.2 and footnote~8 in~\cite{fot} or Théorème~3.1 in~\cite{bh}). Here $\ph'$ and $\psi'$ denote the derivatives of $\ph$ and $\psi$ whenever they exist, extended arbitrarily to Borel functions on all of $\R$.

Recall that $T_t$ denotes the Markovian semigroup associated with the Dirichlet form $\E$; the relation between the form $\E$ and the semigroup $T_t$ is given by~\eqref{eq:form:heat}. For $u \in L^2(E)$, we have
\begin{align}
\label{eq:heat}
 T_t u(x) & = \int_E u(y) T_t(x, dy)
\end{align}
(with equality almost everywhere) for an appropriate kernel $T_t(x, dy)$, and we sometimes write $T_t(dx, dy) = T_t(x, dy) m(dx)$. The semigroup $T_t$ is Markovian if its kernel $T_t(x, dy)$ is sub-probabilistic, meaning it is nonnegative and satisfies $T_t(x, E) \le 1$ for almost every $x \in E$. The operators $T_t$ are self-adjoint, and hence the kernel $T_t(x, dy)$ is symmetric, that is, $T_t(dx, dy) = T_t(dy, dx)$. For $p \in [1, \infty]$, formula~\eqref{eq:heat} extends the definition of $T_t u$ to arbitrary $u \in L^p(E)$, and by Jensen's inequality and Fubini's theorem we have $\norm{T_t u}_p \le \norm{u}_p$. In other words, $T_t$ are contractions on $L^p(E)$. We refer to Section~1.4 in~\cite{fot} for further discussion.

For $p \in (1, \infty)$, we define the Sobolev--Bregman form $\E_p$ corresponding to $\E$ by~\eqref{eq:pform:heat}; that is, we set
\begin{align*}
 \E_p(u) & = \lim_{t \to 0^+} \frac{1}{t} \int_E (u(x) - T_t u(x)) u\pow{p - 1}(x) m(dx)
\end{align*}
whenever the finite limit exists, and in this case we write $u \in \D(\E_p)$. We denote the right-hand side of the Beurling--Deny formula~\eqref{eq:pform} for $\E_p$ by $\Ea_p$, so that the core of Theorem~\ref{thm:main} can be rephrased as follows: if $u \in \D(\E_p)$ and $\tilde{u}$ is the quasi-continuous modification of $u$, then $\E_p(u) = \Ea_p(\tilde{u})$. More precisely, we define
\begin{align}
\label{eq:pform:alt}
 \Ea_p(u) & = \Eal_p(u) + \Eaj_p(u) + \Eak_p(u) ,
\end{align}
where $\Eal_p$, $\Eaj_p$ and $\Eak_p$ correspond to the strongly local part, the purely nonlocal part, and the killing part in~\eqref{eq:pform}, respectively:
\begin{align*}
 \Eal_p(u) & = \frac{4 (p - 1)}{p^2} \, \El(u\pow{p/2}) , \\
 \Eaj_p(u) & = \frac{1}{2} \iint\limits_{(E \times E) \setminus \Delta} (u(y) - u(x)) (u\pow{p - 1}(y) - u\pow{p - 1}(x)) J(dx, dy) , \\
 \Eak_p(u) & = \int_E \abs{u(x)}^p k(dx) .
\end{align*}
Note that $\Eaj_p(u)$ and $\Eak_p(u)$ are well-defined (possibly infinite) for an arbitrary Borel function $u$, because the integrands on the right-hand sides are nonnegative. We stress that here we cannot identify functions $u$ equal almost everywhere, because $k$ and $J$ can charge sets of zero measure $m$ or $m \times m$. The strongly local part $\Eal_p(u)$ is defined whenever $u\pow{p/2} \in \D(\El)$. We do not specify the domain of $\Ea_p$; in particular, we never use the symbol $\D(\Ea_p)$ below.

For $p = 2$, we recover the original Dirichlet form: $\E(u) = \E_2(u) = \Ea_2(\tilde{u})$ whenever $u \in \D(\E)$ and $\tilde{u}$ is the quasi-continuous modification of $u$.

For $p \in (1, \infty)$ and $u \in L^p(E)$, we define the approximate Sobolev--Bregman form
\begin{align*}
 \E_p^{(t)}(u) & = \frac{1}{t} \int_E (u(x) - T_t u(x)) u\pow{p - 1}(x) m(dx) .
\end{align*}
Note that this is the expression under the limit in the definition~\eqref{eq:pform:heat} of $\E_p(u)$.

If $u \in L^p(E)$, then $T_t u \in L^p(E)$, and hence the above integral is well-defined and finite (by Hölder's inequality). Let $\ind$ denote the constant function defined by $\ind(x) = 1$ for $x \in E$. Since $T_t$ is Markovian, we have $0 \le T_t \ind \le 1$ almost everywhere. Clearly,
\begin{align*}
 \E_p^{(t)}(u) & = \frac{1}{t} \int_E (u(x) T_t \ind(x) - T_t u(x)) u\pow{p - 1}(x) m(dx) \\
 & \qquad\qquad + \frac{1}{t} \int_E (1 - T_t \ind(x)) \abs{u(x)}^p m(dx) \\
 & = \frac{1}{t} \int_E \biggl( \int_E (u(x) - u(y)) u\pow{p - 1}(x) T_t(x, dy) \biggr) m(dx) \\
 & \qquad\qquad + \frac{1}{t} \int_E (1 - T_t \ind(x)) \abs{u(x)}^p m(dx)
\end{align*}
By Young's inequality, $(\abs{u(x)} + \abs{u(y)}) \abs{u(x)}^{p - 1}$ is integrable with respect to $T_t(dx, dy) = T_t(x, dy) m(dx)$, and hence
\begin{align*}
 \E_p^{(t)}(u) & = \frac{1}{t} \iint\limits_{E \times E} (u(x) - u(y)) u\pow{p - 1}(x) T_t(dx, dy) \\
 & \qquad\qquad + \frac{1}{t} \int_E (1 - T_t \ind(x)) \abs{u(x)}^p m(dx) .
\end{align*}
Using the symmetry of $T_t$, we find that
\begin{align}
\label{eq:pform:approx}
 \begin{aligned}
 \E_p^{(t)}(u) & = \frac{1}{2 t} \iint\limits_{E \times E} (u(x) - u(y)) (u\pow{p - 1}(x) - u\pow{p - 1}(y)) T_t(dx, dy) \\
 & \qquad\qquad + \frac{1}{t} \int_E (1 - T_t \ind(x)) \abs{u(x)}^p m(dx) ,
 \end{aligned}
\end{align}
and in particular both integrals on the right-hand side are finite.

We recall that for $p = 2$, the spectral theorem implies that for an arbitrary $u \in L^2(E)$, the approximate Dirichlet form $\E_2^{(t)}(u)$ is a nonincreasing function of $t \in (0, \infty)$, and the limit of $\E_2^{(t)}(u)$ as $t \to 0^+$ is equal to $\E(u)$, whether finite or not (see Lemma~1.3.4 in~\cite{fot}). If $\E(u) < \infty$, then we have already noted that $u$ has a quasi-continuous modification $\tilde{u}$, and the Beurling--Deny formula says that $\E(u) = \Ea_2(\tilde{u})$, with $\Ea_2$ defined in~\eqref{eq:pform:alt}.

For $p = 2$, we will also need the approximate bilinear form $\E_2^{(t)}(u, v)$, defined by
\begin{align*}
 \E_2^{(t)}(u, v) & = \frac{1}{2 t} \iint\limits_{E \times E} (u(y) - u(x)) (v(y) - v(x)) T_t(dx, dy) \\
 & \qquad\qquad + \frac{1}{t} \int_E u(x) v(x) (1 - T_t \ind(x)) m(dx) .
\end{align*}
By Lemma~1.3.4 in~\cite{fot}, $\E_2^{(t)}(u, v)$ converges to $\E(u, v)$ as $t \to 0^+$ whenever $u, v \in \D(\E)$.

This last property shows that the Dirichlet form can be approximated using the kernel $T_t(dx, dy)$. We will need a more detailed result. By Lemmas~4.5.2 and~4.5.3 in~\cite{fot} (see also equation~(1.4) in~\cite{cf}), for $u, v \in \D(\E)$ we have
\begin{align}
\label{eq:killing:approx}
 \lim_{t \to 0^+} \frac{1}{t} \int_E u(x) v(x) (1 - T_t \ind(x)) m(dx) & = \int_E \tilde{u}(x) \tilde{v}(x) k(dx) ,
\end{align}
so that the killing part can be recovered from the kernel $T_t(dx, dy)$. Combining this with the convergence of $\E_2^{(t)}(u, v)$ to $\E(u, v)$ and the expressions~\eqref{eq:pform:approx} for $\E_2^{(t)}(u, v)$ and~\eqref{eq:form:quasi} for $\E(u, v)$, we find that
\begin{align}
\label{eq:jump:approx}
 \begin{aligned}
 & \lim_{t \to 0^+} \frac{1}{2 t} \iint\limits_{E \times E} (u(y) - u(x)) (v(y) - v(x)) T_t(dx, dy) \\
 & \qquad\qquad = \El(u, v) + \frac{1}{2} \iint\limits_{(E \times E) \setminus \Delta} (\tilde{u}(y) - \tilde{u}(x)) (\tilde{v}(y) - \tilde{v}(x)) J(dx, dy) .
 \end{aligned}
\end{align}
The above identity expresses the form $\E$ with the killing part removed (the \emph{resurrected form}) in terms of the kernel $T_t(dx, dy)$.

We conclude this section with the following observation. If $u \in \D(\E)$, $C > 0$, and $v$ is a normal contraction of $C u$ (that is, $\abs{v(x)} \le C \abs{u(x)}$ and $\abs{v(x) - v(y)} \le C \abs{u(x) - u(y)}$), then $v \in \D(\E)$ and $\E(v, v) \le C^2 \E(u, u)$. In particular, if $v(x) = \ph(u(x))$ for a Lipschitz function $\ph$ such that $\ph(0) = 0$, then $v \in \D(\E)$.

The properties listed above will often be used without further comments.

%
%

\section{Proof of the main result}
\label{sec:proof}

For convenience, we divide the proof of Theorem~\ref{thm:main} into three parts.

\begin{proof}[Proof of Theorem~\ref{thm:main}, part I]
We first prove the final claim of the theorem, with a minor modification: we replace the condition $u \in \D(\E_p)$ by $u\pow{p/2} \in \D(\E)$. Specifically, we prove two statements: (a)~if $u\pow{p/2} \in \D(\E)$, then $u$ has a quasi-continuous modification $\tilde{u}$ such that the integrals in~\eqref{eq:pform} are finite; (b)~conversely, if $\E$ is maximally defined, $u$ has a quasi-continuous modification $\tilde{u}$ and the integrals in~\eqref{eq:pform} are finite, then $u\pow{p/2} \in \D(\E)$. In the next two parts of the proof, we show that the two conditions $u \in \D(\E_p)$ and $u\pow{p/2} \in \D(\E)$ are equivalent, and this completes the proof of the final part of the theorem.

Let $u$ be a Borel function. Observe that $u \in L^p(E)$ if and only if $u\pow{p/2} \in L^2(E)$, and $u$ is quasi-continuous if and only if $u\pow{p/2}$ is quasi-continuous. By Lemma~\ref{lem:forms}, applied to $s = u\pow{p/2}(x)$, $t = u\pow{p/2}(y)$ and $\alpha = \frac{2}{p}$, we have
\begin{align}
\label{eq:forms:u}
 \begin{aligned}
 & \frac{4 (p - 1)}{p^2} \, (u\pow{p/2}(y) - u\pow{p/2}(x))^2 \\
 & \qquad\qquad \le (u(y) - u(x)) (u\pow{p - 1}(y) - u\pow{p - 1}(x)) \\
 & \qquad\qquad\qquad\qquad \le 2 (u\pow{p/2}(y) - u\pow{p/2}(x))^2 .
 \end{aligned}
\end{align}
Integrating both sides with respect to $J(dx, dy)$, we find that
\begin{align*}
 \frac{4 (p - 1)}{p^2} \, \Eaj_2(u\pow{p/2}) \le \Eaj_p(u) & \le 2 \Eaj_2(u\pow{p/2}) .
\end{align*}
In particular, $\Eaj_p(u)$ is finite if and only if $\Eaj_2(u\pow{p/2})$ is finite. Furthermore, by definition,
\begin{align*}
 \Eal_p(u) & = \frac{4 (p - 1)}{p^2} \, \Eal_2(u\pow{p/2}) , \\
 \Eak_p(u) & = \Eak_2(u\pow{p/2}) .
\end{align*}
Recall that $\Ea_p(u) = \Eal_p(u) + \Eaj_p(u) + \Eak_p(u)$.

Suppose now that $u\pow{p/2} \in \D(\E)$. Then $u\pow{p/2}$ has a quasi-continuous modification $\tilde{u}\pow{p/2}$, and $\Ea_2(\tilde{u}\pow{p/2}) = \E(u\pow{p/2})$ is finite. This implies that $\tilde{u}$ is a quasi-continuous modification of $u$, and the integrals $\Eaj_p(\tilde{u}) \le 2 \Eaj_2(\tilde{u}\pow{p/2})$ and $\Eak_p(\tilde{u}) = \Eak_2(\tilde{u}\pow{p/2})$ are finite. We have thus proved that the integrals in~\eqref{eq:pform} are finite, as desired.

Conversely, suppose that $\E$ is maximally defined, $u$ has a quasi-continuous modification $\tilde{u}$ and the integrals in~\eqref{eq:pform} are finite; that is, $\Eaj_p(\tilde{u})$ and $\Eak_p(\tilde{u})$ are finite. Then we find that $\Eaj_2(\tilde{u}\pow{p/2}) \le (p^2 / (4 (p - 1))) \Eaj_p(\tilde{u})$ and $\Eak_2(\tilde{u}\pow{p/2}) = \Eak_p(\tilde{u})$ are finite. By our maximality assumption on $\D(\E)$, this implies that $\tilde{u}\pow{p/2} \in \D(\E)$, and hence $u\pow{p/2} \in \D(\E)$.
\end{proof}

\begin{proof}[Proof of Theorem~\ref{thm:main}, part II]
We now prove that if $u \in \D(\E_p)$, then $u\pow{p/2} \in \D(\E)$, and that $\E_p(u)$ is comparable with $\E(u\pow{p/2})$.

Let $u \in L^p(E)$. As in part~I of the proof, by~\eqref{eq:forms:u} and~\eqref{eq:pform:approx} we have
\begin{align*}
 \frac{4 (p - 1)}{p^2} \, \E_2^{(t)}(u\pow{p/2}) \le \E_p^{(t)}(u) & \le 2 \E_2^{(t)}(u\pow{p/2}) .
\end{align*}
Suppose that $u \in \D(\E_p)$, that is, a finite limit of $\E_p^{(t)}(u)$ as $t \to 0^+$ exists. The above estimate implies that $\E_2^{(t)}(u\pow{p/2})$ is bounded as $t \to 0^+$. Since $\E_2^{(t)}(u\pow{p/2})$ is a nonincreasing function of $t \in (0, \infty)$, we conclude that a finite limit of $\E_2^{(t)}(u\pow{p/2})$ as $t \to 0^+$ exists, and hence $u\pow{p/2} \in \D(\E)$. Additionally, \eqref{eq:pform:comparability} follows: we have
\begin{align*}
 \frac{4 (p - 1)}{p^2} \, \E_2(u\pow{p/2}) \le \E_p(u) & \le 2 \E_2(u\pow{p/2}) . \qedhere
\end{align*}
\end{proof}

\begin{proof}[Proof of Theorem~\ref{thm:main}, part III]
In this final part we show that if $u\pow{p/2} \in \D(\E)$, then $u \in \D(\E_p)$ and $\E_p(u)$ is given by the Beurling--Deny formula~\eqref{eq:pform}.

Let $u \in L^p(E)$, and suppose that $u\pow{p/2} \in \D(\E)$. In this case the function $u\pow{p/2}$ has a quasi-continuous modification, and so also $u$ has a quasi-continuous modification. For simplicity, throughout the proof we denote this modification again by $u$. Our goal is to show that $u \in \D(\E_p)$ and $\E_p(u) = \Ea_p(u)$. Recall that the Sobolev--Bregman form $\E_p(u)$ is the limit of the approximate forms $\E_p^{(t)}(u)$ as $t \to 0^+$, and that $\E_p^{(t)}(u)$ is given by~\eqref{eq:pform:approx}. On the other hand, $\Ea_p(u) = \Eal_p(u) + \Eaj_p(u) + \Eak_p(u)$. Thus, we need to prove that
\begin{align*}
 & \lim_{t \to 0^+} \biggl( \frac{1}{2 t} \iint\limits_{E \times E} (u(y) - u(x)) (u\pow{p - 1}(y) - u\pow{p - 1}(x)) T_t(dx, dy) \\
 & \qquad\qquad\qquad\qquad\qquad\qquad\qquad\qquad + \frac{1}{t} \int_E \abs{u(x)}^p (1 - T_t \ind(x)) m(dx) \biggr) \\
 & \qquad = \Eal_p(u) + \Eaj_p(u) + \Eak_p(u) ,
\end{align*}
and that the right-hand side is finite.

By~\eqref{eq:killing:approx} with $u$ and $v$ replaced by $u\pow{p/2}$, a finite limit
\begin{align*}
 \lim_{t \to 0^+} \frac{1}{t} \int_E \abs{u(x)}^p (1 - T_t \ind(x)) m(dx) & = \int_E \abs{u(x)}^p k(dx) = \Eak_p(u)
\end{align*}
exists. Therefore, it is enough to show that
\begin{align}
\label{eq:claim}
 \lim_{t \to 0^+} \frac{1}{2 t} \iint\limits_{E \times E} (u(y) - u(x)) (u\pow{p - 1}(y) - u\pow{p - 1}(x)) T_t(dx, dy) & = \Eal_p(u) + \Eaj_p(u) .
\end{align}
This is the most technical part of the proof of Theorem~\ref{thm:main}, and we divide the argument into six steps.

\emph{Step 1.} Let $\alpha = \frac{2}{p}$ and $v = u\pow{p/2}$, so that $\alpha \in (0, 2)$, $v \in \D(\E)$ and $v$ is quasi-continuous. In terms of $\alpha$ and $v$, we have
\begin{align*}
 \Eal_p(u) & = \frac{4 (p - 1)}{p^2} \, \El(u\pow{p/2}) = \alpha (2 - \alpha) \El(v)
\end{align*}
and
\begin{align*}
 \Eaj_p(u) & = \frac{1}{2} \iint\limits_{(E \times E) \setminus \Delta} (u(y) - u(x)) (u\pow{p - 1}(y) - u\pow{p - 1}(x)) J(dx, dy) \\
 & = \frac{1}{2} \iint\limits_{(E \times E) \setminus \Delta} (v\pow{\alpha}(y) - v\pow{\alpha}(x)) (v\pow{2 - \alpha}(y) - v\pow{2 - \alpha}(x)) J(dx, dy) .
\end{align*}
Thus, our goal~\eqref{eq:claim} reads
\begin{align}
\label{eq:claim:v}
\begin{aligned}
 & \lim_{t \to 0^+} \frac{1}{2 t} \iint\limits_{E \times E} (v\pow{\alpha}(y) - v\pow{\alpha}(x)) (v\pow{2 - \alpha}(y) - v\pow{2 - \alpha}(x)) T_t(dx, dy) \\
 & \quad = \alpha (2 - \alpha) \El(v) + \frac{1}{2} \iint\limits_{(E \times E) \setminus \Delta} (v\pow{\alpha}(y) - v\pow{\alpha}(x)) (v\pow{2 - \alpha}(y) - v\pow{2 - \alpha}(x)) J(dx, dy) .
\end{aligned}
\end{align}
We will apply \eqref{eq:main} for $s = n^2 v(x)$ and $t = n^2 v(y)$, where $n = 2, 3, \ldots$\, Using the functions $\ph_\alpha$, $\ph_{\alpha, n}$ and $\psi_n$ introduced in Lemma~\ref{lem:main}, let us denote
\begin{align*}
 v_\alpha(x) & = n^{-2 \alpha} \ph_\alpha(n^2 v(x)) = v\pow{\alpha}(x) , \\
 v_{\alpha, n}(x) & = n^{-2 \alpha} \ph_{\alpha, n}(n^2 v(x)) , \\
 w_n(x) & = n^{-2} \psi_n(n^2 v(x)) ,
\end{align*}
and furthermore
\begin{align*}
 \Phi_\alpha(x, y) & = (v_\alpha(y) - v_\alpha(x)) (v_{2 - \alpha}(y) - v_{2 - \alpha}(x)) , \\
 \Phi_{\alpha, n}(x, y) & = (v_{\alpha, n}(y) - v_{\alpha, n}(x)) (v_{2 - \alpha, n}(y) - v_{2 - \alpha, n}(x)) .
\end{align*}
Finally, let $c = 180$ and $\eps_{\alpha, n} = 8 n^{-\min\{\alpha, 2 - \alpha\}}$. With this notation, Lemma~\ref{lem:main} asserts that
\begin{align}
\label{eq:main:v}
 \abs{\Phi_\alpha(x, y) - \Phi_{\alpha, n}(x, y)} & \le \eps_{\alpha, n} (v(y) - v(x))^2 + c (w_n(y) - w_n(x))^2 ,
\end{align}
while our goal~\eqref{eq:claim:v} takes form
\begin{align}
\label{eq:main:goal}
 \lim_{t \to 0^+} \frac{1}{2 t} \iint\limits_{E \times E} \Phi_\alpha(x, y) T_t(dx, dy) & = \alpha (2 - \alpha) \El(v) + \frac{1}{2} \iint\limits_{(E \times E) \setminus \Delta} \Phi_\alpha(x, y) J(dx, dy) .
\end{align}
In order to prove this equality, we approximate $\Phi_\alpha$ by $\Phi_{\alpha, n}$ and estimate the error.

\emph{Step 2.} Since $v \in \D(\E)$ and $v$ is quasi-continuous, we have, by~\eqref{eq:jump:approx},
\begin{align}
\label{eq:main:d}
 \lim_{t \to 0^+} \frac{1}{2 t} \iint\limits_{E \times E} (v(y) - v(x))^2 T_t(dx, dy) & = \El(v) + \frac{1}{2} \iint\limits_{E \times E} (v(y) - v(x))^2 J(dx, dy) .
\end{align}

\emph{Step 3.} We have $v \in \D(\E)$, and $\ph_{\alpha, n}$ is a Lipschitz function satisfying $\ph_{\alpha, n}(0) = 0$. Thus, the function $v_{\alpha, n}(x) = n^{-2 \alpha} \ph_{\alpha, n}(n^2 v(x))$ is in $\D(\E)$. Similarly, $v_{2 - \alpha, n} \in \D(\E)$. Additionally, $v$, $v_{\alpha, n}$, and $v_{2 - \alpha, n}$ are quasi-continuous. By~\eqref{eq:jump:approx},
\begin{align*}
 & \lim_{t \to 0^+} \frac{1}{2 t} \iint\limits_{E \times E} (v_{\alpha, n}(y) - v_{\alpha, n}(x)) (v_{2 - \alpha, n}(y) - v_{2 - \alpha, n}(x)) T_t(dx, dy) \\
 & \qquad = \El(v_{\alpha, n}, v_{2 - \alpha, n}) + \frac{1}{2} \iint\limits_{(E \times E) \setminus \Delta} (v_{\alpha, n}(y) - v_{\alpha, n}(x)) (v_{2 - \alpha, n}(y) - v_{2 - \alpha, n}(x)) J(dx, dy).
\end{align*}
In terms of $\Phi_{\alpha, n}$, this equality reads
\begin{align}
\label{eq:main:a}
\begin{aligned}
 & \lim_{t \to 0^+} \frac{1}{2 t} \iint\limits_{E \times E} \Phi_{\alpha, n}(x, y) T_t(dx, dy) \\
 & \qquad = \El(v_{\alpha, n}, v_{2 - \alpha, n}) + \frac{1}{2} \iint\limits_{(E \times E) \setminus \Delta} \Phi_{\alpha, n}(x, y) J(dx, dy) .
\end{aligned}
\end{align}

\emph{Step 4.} Since $v \in \D(\E)$ and $\psi_n$ are Lipschitz functions, as in the previous step we find that the functions $w_n(x) = n^{-2} \psi_n(n^2 v(x))$ are in $\D(\E)$. Since $v$ and $w_n$ are quasi-continuous, from~\eqref{eq:jump:approx} we obtain
\begin{align}
\label{eq:main:b}
\begin{aligned}
 & \lim_{t \to 0^+} \frac{1}{2 t} \iint\limits_{E \times E} (w_n(y) - w_n(x))^2 T_t(dx, dy) \\
 & \qquad = \El(w_n) + \frac{1}{2} \iint\limits_{(E \times E) \setminus \Delta} (w_n(y) - w_n(x))^2 J(dx, dy) .
\end{aligned}
\end{align}
Furthermore, the functions $s \mapsto n^{-2} \psi_n(n^2 s)$ have Lipschitz constant $1$ and they converge pointwise to zero as $n \to \infty$. Thus, $w_n(x) = n^{-2} \psi_n(n^2 v(x))$ converge pointwise to zero as $n \to \infty$, and additionally $\abs{w_n(y) - w_n(x)} \le \abs{v(y) - v(x)}$. Since $(v(y) - v(x))^2$ is integrable with respect to $J(dx, dy)$, we may use the dominated convergence theorem to find that
\begin{align}
\label{eq:main:c}
 \lim_{n \to \infty} \iint\limits_{(E \times E) \setminus \Delta} (w_n(y) - w_n(x))^2 J(dx, dy) & = 0 .
\end{align}

\emph{Step 5.} The above results are sufficient to handle the purely nonlocal part, and we now turn to the properties of the strongly local part. Recall that $w_n(x) = n^{-2} \psi_n(n^2 v(x))$, and
\begin{align*}
 \psi_n'(s) & =
 \begin{cases}
  1 & \text{when $\abs{s} < n$,} \\
  0 & \text{when $n < \abs{s} < n^3$,} \\
  1 & \text{when $\abs{s} > n^3$.}
 \end{cases}
\end{align*}
By~\eqref{eq:form:local}, we have
\begin{align*}
 \El(w_n) & = \int_E (\psi_n'(n^2 v(x)))^2 \mul_v(dx) \\
 & = \int_E \ind_{(0, 1/n) \cup (n, \infty)}(\abs{v(x)}) \mul_v(dx) .
\end{align*}
Hence, by the dominated convergence theorem,
\begin{align}
\label{eq:main:e}
 \lim_{n \to \infty} \El(w_n) & = 0 .
\end{align}
Similarly, we have $v_{\alpha, n}(x) = n^{-2} \ph_{\alpha, n}(n^2 v(x))$ and $v_{2 - \alpha, n}(x) = n^{-2} \ph_{2 - \alpha, n}(n^2 v(x))$. Thus, again by~\eqref{eq:form:local},
\begin{align*}
 \El(v_{\alpha, n}, v_{2 - \alpha, n}) & = \int_E \ph_{\alpha, n}'(n^2 v(x)) \ph_{2 - \alpha, n}'(n^2 v(x)) \mul_v(dx) .
\end{align*}
However,
\begin{align*}
 \ph_{\alpha, n}'(s) & =
 \begin{cases}
  1 & \text{when $\abs{s} < 1$,} \\
  \alpha \abs{s}^{\alpha - 1} & \text{when $1 < \abs{s} < n^4$,} \\
  0 & \text{when $\abs{s} > n^4$.}
 \end{cases}
\end{align*}
Thus,
\begin{align*}
 \ph_{\alpha, n}'(s) \ph_{2 - \alpha, n}'(s) & =
 \begin{cases}
  1 & \text{when $\abs{s} < 1$,} \\
  \alpha (2 - \alpha) & \text{ when $1 < \abs{s} < n^4$,} \\
  0 & \text{when $\abs{s} > n^4$,}
 \end{cases}
\end{align*}
and it follows that
\begin{align*}
 \El(v_{\alpha, n}, v_{2 - \alpha, n}) & = \int_E \Bigl(\ind_{(0, 1/n^2)}(\abs{v(x)}) + \alpha (2 - \alpha) \ind_{(1/n^2, n^2)}(\abs{v(x)})\Bigr) \mul_v(dx) .
\end{align*}
Using the dominated convergence theorem, we find that
\begin{align}
\label{eq:main:f}
 \lim_{n \to \infty} \El(v_{\alpha, n}, v_{2 - \alpha, n}) & = \int_E \alpha (2 - \alpha) \mul_v(dx) = \alpha (2 - \alpha) \El(v) .
\end{align}

\emph{Step 6.} In order to prove our goal~\eqref{eq:main:goal}, we denote
\begin{align*}
 L & = \limsup_{t \to 0^+} \biggl| \frac{1}{2 t} \iint\limits_{E \times E} \Phi_\alpha(x, y) T_t(dx, dy) - \alpha (2 - \alpha) \El(v) - \frac{1}{2}\iint\limits_{(E \times E) \setminus \Delta} \Phi_\alpha(x, y) J(dx, dy) \biggr| .
\end{align*}
We claim that $L = 0$. Clearly, for $t > 0$ and $n = 2, 3, \ldots$ we have
\begin{align*}
 & \biggl| \frac{1}{2 t} \iint\limits_{E \times E} \Phi_\alpha(x, y) T_t(dx, dy) - \alpha (2 - \alpha) \El(v) - \frac{1}{2} \iint\limits_{(E \times E) \setminus \Delta} \Phi_\alpha(x, y) J(dx, dy) \biggr| \\
 & \qquad \le \biggl| \frac{1}{2 t} \iint\limits_{E \times E} \Phi_{\alpha, n}(x, y) T_t(dx, dy) - \El(v_{\alpha, n}, v_{2 - \alpha, n}) - \frac{1}{2} \iint\limits_{(E \times E) \setminus \Delta} \Phi_{\alpha, n}(x, y) J(dx, dy) \biggr| \\
 & \qquad\qquad + \frac{1}{2 t} \iint\limits_{E \times E} \abs{\Phi_\alpha(x, y) - \Phi_{\alpha, n}(x, y)} T_t(dx, dy) \\
 & \qquad\qquad\qquad + \abs{\El(v_{\alpha, n}, v_{2 - \alpha, n}) - \alpha (2 - \alpha) \El(v)} \\
 & \qquad\qquad\qquad\qquad + \frac{1}{2} \iint\limits_{(E \times E) \setminus \Delta} \abs{\Phi_{\alpha, n}(x, y) - \Phi_\alpha(x, y)} J(dx, dy) .
\end{align*}
By~\eqref{eq:main:a}, the first term on the right-hand side converges to zero as $t \to 0^+$ for every $n = 2, 3, \ldots$\, Therefore,
\begin{align*}
 L & \le \limsup_{t \to 0^+} \frac{1}{2 t} \iint\limits_{E \times E} \abs{\Phi_\alpha(x, y) - \Phi_{\alpha, n}(x, y)} T_t(dx, dy) \\
 & \qquad + \abs{\El(v_{\alpha, n}, v_{2 - \alpha, n}) - \alpha (2 - \alpha) \El(v)} \\
 & \qquad\qquad + \frac{1}{2} \iint\limits_{(E \times E) \setminus \Delta} \abs{\Phi_{\alpha, n}(x, y) - \Phi_\alpha(x, y)} J(dx, dy) .
\end{align*}
We apply~\eqref{eq:main:v} to each of the integrands on the right-hand side to find that
\begin{align*}
 L & \le \limsup_{t \to 0^+} \biggl(\frac{\eps_{\alpha, n}}{2 t} \iint\limits_{E \times E} (v(y) - v(x))^2 T_t(dx, dy) + \frac{c}{2 t} \iint\limits_{E \times E} (w_n(y) - w_n(x))^2 T_t(dx, dy) \biggr) \\
 & \qquad + \abs{\El(v_{\alpha, n}, v_{2 - \alpha, n}) - \alpha (2 - \alpha) \El(v)} \\
 & \qquad\qquad + \frac{\eps_{\alpha, n}}{2} \iint\limits_{(E \times E) \setminus \Delta} (v(y) - v(x))^2 J(dx, dy) + \frac{c}{2} \iint\limits_{(E \times E) \setminus \Delta} (w_n(y) - w_n(x))^2 J(dx, dy) .
\end{align*}
Next, we use~\eqref{eq:main:d} and~\eqref{eq:main:b} to transform the first two terms on the right-hand side (and combine them with the last two terms):
\begin{align*}
 L & \le \eps_{\alpha, n} \El(v) + \eps_{\alpha, n} \iint\limits_{(E \times E) \setminus \Delta} (v(y) - v(x))^2 J(dx, dy) \\
 & \qquad + c \El(w_n) + c \iint\limits_{(E \times E) \setminus \Delta} (w_n(y) - w_n(x))^2 J(dx, dy) \\
 & \qquad\qquad + \abs{\El(v_{\alpha, n}, v_{2 - \alpha, n}) - \alpha (2 - \alpha) \El(v)} .
\end{align*}
The above estimate holds for every $n = 2, 3, \ldots$ and the left-hand side does not depend on $n$. The first two terms on the right-hand side converge to zero as $n \to \infty$ by the definition of $\eps_{\alpha, n}$. The third term tends to zero by~\eqref{eq:main:e}, the fourth one by~\eqref{eq:main:c}, and the fifth one by~\eqref{eq:main:f}. Thus, the left-hand side $L$ is necessarily zero, as claimed. We have thus proved our goal~\eqref{eq:main:goal}, or, equivalently, that $u \in \D(\E_p)$ and that $\E_p(u)$ is given by the analogue of the Beurling--Deny formula~\eqref{eq:claim}.
\end{proof}

%
%

\section{Example: Euclidean spaces}
\label{sec:example}

Suppose that $E$ is a domain in a Euclidean space $\R^n$, and that $\E$ is a regular Dirichlet form such that all smooth, compactly supported functions on $E$ belong to $\D(\E)$. In this case, the strongly local part admits a more explicit description: for all smooth, compactly supported functions $u, v$ on $E$ we have
\begin{align}
\notag
 \E(u, v) & = \int_E \sum_{i, j = 1}^n \frac{\partial u}{\partial x_i}(x) \frac{\partial v}{\partial x_j}(x) \nu_{i, j}(dx) \\
\label{eq:ex:form}
 & \qquad + \frac{1}{2} \iint\limits_{(E \times E) \setminus \Delta} (u(y) - u(x)) (v(y) - v(x)) J(dx, dy) \\
\notag
 & \qquad\qquad + \int_E u(x) v(x) k(dx)
\end{align}
for some locally finite measures $\nu_{i, j}$ on $E$ (see Theorem~3.2.3 in~\cite{fot}). In this case, by Theorem~\ref{thm:main}, the Sobolev--Bregman form is given explicitly by 
\begin{align}
\notag
 \E_p(u) & = (p - 1) \int_E \sum_{i, j = 1}^n \abs{u(x)}^{p - 2} \frac{\partial u}{\partial x_i}(x) \frac{\partial u}{\partial x_j}(x) \nu_{i, j}(dx) \\
\label{eq:ex:pform}
 & \qquad + \frac{1}{2} \iint\limits_{(E \times E) \setminus \Delta} (u(y) - u(x)) (u\pow{p - 1}(y) - u\pow{p - 1}(x)) J(dx, dy) \\
\notag
 & \qquad\qquad + \int_E \abs{u(x)}^p k(dx)
\end{align}
for every smooth, compactly supported function $u$ on $E$. Furthermore, whenever~\eqref{eq:ex:form} holds for a broader class $\D$ of admissible functions $u, v$ (e.g.\@ an appropriate Sobolev space $W^{1,2}_0(E)$ or $W^{1,2}(E)$ when the strongly local part of $\E$ corresponds to a uniformly elliptic second order operator), then~\eqref{eq:ex:pform} holds true for every $u$ such that $u\pow{p/2} \in \D$.

%
%

\subsection*{Acknowledgements}

We are grateful to Krzysztof Bogdan for inspiration, discussions, and numerous comments on the preliminary version of this paper. We also thank René L.~Schilling for advice on domains of Dirichlet forms. We thank the anonymous referee for thoughtful suggestions, which greatly improved our manuscript.

%
%

%
%

\end{document}